\documentclass[reqno]{amsart} 

\usepackage{amsfonts, amsmath, amsthm, amssymb}

\def\b{\mathbb }
\def\cal{\mathcal }

\theoremstyle{plain}
\newtheorem{theorem}{Theorem}
\newtheorem{lemma}{Lemma}
\newtheorem{corollary}{Corollary}
\newtheorem{proposition}{Proposition}
\newtheorem*{theorem A}{Theorem A}
\newtheorem*{theorem B}{Theorem B}

\theoremstyle{remark}
\newtheorem{remark}{Remark}

\newtheorem*{remarkex}{Example}

\theoremstyle{definition}

\title[Beta distributions and Sonine integrals]
{Beta distributions and Sonine integrals for Bessel functions on symmetric cones}
\author{Margit R\"osler} 
\address{Institut f\"ur Mathematik, Universit\"at Paderborn, Warburger Str. 100,
D-33098 Paderborn, Germany}
\email{roesler@math.upb.de}
\author{Michael Voit}
\address{Fakult\"at Mathematik, Technische Universit\"at Dortmund,
          Vogelpothsweg 87,
          D-44221 Dortmund, Germany}
\email{michael.voit@math.tu-dortmund.de}


\begin{document}

\maketitle

\begin{abstract}

There exist several multivariate extensions of the classical Sonine
integral representation for Bessel functions of some  index $\mu+ \nu$  
with respect to such functions of lower index $\mu.$  For Bessel functions on matrix cones, Sonine formulas involve beta densities $\beta_{\mu,\nu}$ on the cone and trace already back to Herz. The Sonine representations known so far on symmetric cones are restricted to continuous ranges  $\Re\mu, \Re \nu > \mu_0$, where the involved beta densities are probability measures and the limiting index $\mu_0\geq 0$ depends on the rank of the cone. It is zero only in the one-dimensional case, but larger than zero in all multivariate cases. 

In this paper, we study the extension of Sonine formulas for Bessel functions on symmetric cones to values of $\nu$ below the critical limit $\mu_0$. This is achieved by an analytic 
extension of the involved beta measures as tempered distributions. 
Following  recent ideas by A. Sokal for Riesz distributions on symmetric cones,
we analyze for which indices the obtained beta distributions are
still measures. At the same time, we characterize the 
indices for which a Sonine formula between the related Bessel functions exists.
As for Riesz distributions, there occur gaps in the admissible range of 
indices which are determined by the so-called Wallach set.

\end{abstract}

\section{Introduction}

Consider the one-variable normalized Bessel functions
\[
j_\alpha(z):=\, _0F_1(\alpha+1;-z^2/4) \quad\quad 
(\alpha\in\mathbb C\setminus\{-1,-2,\ldots\}),
\]
which  for $\alpha>-1/2$ have the well-known  Laplace integral representation
\begin{equation}\label{intrep-1-dim-allg-exp}
j_\alpha(z):=\frac{\Gamma(\alpha+1)}{\Gamma(\alpha+1/2)\Gamma(1/2)}
\int_{-1}^1 e^{izx}(1-x^2)^{\alpha-1/2} \> dx \quad\quad(z\in\mathbb C).
\end{equation}

For half integers $\alpha=p/2-1$
 with $p\geq 2$, 
 formula \eqref{intrep-1-dim-allg-exp}
 may be regarded as a Harish-Chandra integral representation 
for the spherical functions of the Euclidean space $\mathbb R^p$
with $SO(p)$-action.
It is also well-known that for $\alpha>-1$ and $\beta>0$, $j_{\alpha+\beta}$ can be expressed
 in terms of $j_{\alpha}$ as a Sonine integral (formula (3.4) in Askey \cite{A}):
\begin{equation}\label{intrep-1-dim-allg}
j_{\alpha+\beta}(z)= 2\frac{\Gamma(\alpha+\beta+1)}{\Gamma(\alpha+1)\Gamma(\beta)}
\int_{0}^1j_{\alpha}(zx)  x^{2\alpha+1}(1-x^2)^{\beta-1}\> dx.
\end{equation}
This follows easily by power series expansion of both sides and is a particular case of classical integral representations for one-variable hypergeometric
functions.  
Notice that for $\beta=0$, formula \eqref{intrep-1-dim-allg} degenerates in a trivial way.
As $j_{-1/2}(z)=\cos z$,
 formula \eqref{intrep-1-dim-allg-exp}
is actually a special case of \eqref{intrep-1-dim-allg}.
For some background on these classical formulas we also refer to the monograph \cite{AAR}.

We now ask for which indices $\alpha, \beta\in \mathbb R$ with $\alpha>-1$ and $ \alpha + \beta >-1$ there actually exists
a Sonine integral representation
$$j_{\alpha + \beta}(z) = \int_0^\infty j_\alpha(zx) d\mu_{\alpha, \beta}(x)$$ 
with a positive measure $\mu_{\alpha,\beta}$. It is easily seen that
this is only possible  if $\beta \geq 0.$ Indeed, if such a representation
with $\beta<0$ would exist, we could combine it with
(\ref{intrep-1-dim-allg}) for the
parameter pairs $(\alpha+\beta, -\beta)$ instead of $(\alpha,\beta)$. This would
lead to a Sonine integral representation 
of $j_\alpha$ in terms of $j_\alpha$ with a measure different from the point measure $\delta_1$, which is
impossible
by the injectivity of the Hankel transform of bounded measures.
In particular, a 
 Laplace representation such as  (\ref{intrep-1-dim-allg-exp}) with a positive representing measure exists
precisely for $\alpha\ge-1/2$.

In this paper we study extensions
of Sonine-type integral representations
for Bessel functions of matrix argument and more generally, on Euclidean Jordan algebras
and the associated symmetric cones. The general Jordan algebra setting includes the Jordan algebras of Hermitian matrices over the (skew) fields $\mathbb R, \mathbb C$ or the quaternions $\mathbb H$ as important special cases. Bessel functions in this setting
trace back to the fundamental work of Herz \cite{H}, which was
motivated by questions in  number theory and multivariate statistics. For example, Bessel functions of matrix argument  occur naturally  in the explicit expression of non-central
Wishart distributions (\cite{Co}, \cite{Mu}). They 
are imbedded in a theory of hypergeometric functions on Euclidean Jordan algebras  which are defined as hypergeometric series in terms of so-called spherical polynomials.
Integral representations of Bessel functions play an important role in the analysis on symmetric cones and are closely related to  Laplace transforms. For details and a general background see \cite{FK, GR, H}. For various aspects concerning the rich harmonic analysis associated with Bessel functions on symmetric cones, we also refer to \cite{FT, Di, R2, RV3, Mo}. 

Let us now describe our results in more detail. In order to avoid abstract notation, we restrict in this introduction to the case where 
the underlying Jordan algebra  is the space $ V= H_q(\mathbb F)$ of $q\times q$ Hermitian
 matrices over $\mathbb F=  \mathbb R, \mathbb C$
or $\mathbb H.$ The (real) dimension of  $V$ is  $$ n=q+ \frac{d}{2}q(q-1) \text{ wih }\, d= dim_{\mathbb R}\mathbb F\in \{1,2,4\},$$ and 
 $V$ is associated with the symmetric cone $\Omega=\Omega_q(\mathbb F)$ of positive definite matrices over $\mathbb F$. 
 The Bessel functions on $V$ are defined by
$$ \mathcal J_\mu(x) := \- _0F_1(\mu; -x) = 
 \sum_{\lambda\geq 0} \frac{(-1)^{|\lambda|}}{(\mu)_\lambda |\lambda|!}
Z_\lambda(x),\quad x \in V.$$
Here the sum is over all partitions of length $q$,  $(\mu)_\lambda$ is a generalized Pochhammer symbol, and the $Z_\lambda$ are the (renormalized) 
spherical polynomials of $\Omega$,  see Section \ref{sec2} for 
the details. 

For 
indices $\mu,\nu\in\mathbb C$ with
$\Re \mu, \Re \nu > n/q-1=:\mu_0$, the associated beta measure on $\Omega$ is defined by

\begin{equation}\label{betameasure}
d\beta_{\mu,\nu}(x):= \frac{1}{ B_\Omega(\mu,\nu)}  \Delta(x)^{\mu-n/q}\Delta(I_q-x)^{\nu-n/q}
\cdot 1_{\Omega_e}(x)dx\end{equation}
where $B_\Omega(\mu,\nu)$ is the beta function associated with  $\Omega$,
$\Delta$ denotes the determinant polynomial on $V,$ and 
$\Omega_e = \{x\in \Omega: x< I_q\}.$ We also consider $\beta_{\mu, \nu}$ 
as a compactly supported measure on $\overline\Omega$ or on $V$. 
For real $\mu, \nu >\mu_0$, 
$\beta_{\mu,\nu}$ is a probability measure. 
It is known for $H_q(\mathbb R)$ and easy to see in the general Jordan setting (Theorem \ref{int-rep-classical}) that the Bessel functions $\mathcal J_\mu$  have 
the following integral representation of Sonine type generalizing the one-variable case \eqref{intrep-1-dim-allg}: 
For indices $\mu, \nu \in \mathbb C$ with $\Re\mu, \Re \nu > \mu_0$,
\begin{equation}\label{Sonine1} \mathcal J_{\mu + \nu}(r) = \int_{\Omega_e} \mathcal J_\mu(\sqrt r s\sqrt r) d\beta_{\mu,\nu}(s) \quad \text{ for all } r\in \Omega\,.\end{equation}
Notice that $\mu_0 = 0$ if $q=1$, but $\mu_0$ is larger than zero if $q>1$, and in this case formula \eqref{Sonine1}  is not available in the range $\Re\nu >0$. This is to some extent unexpected and makes the situation
more interesting in higher dimension than in the one-variable case. 

Let us mention at this point that there is a broad literature on beta probability distributions on matrix cones and their relevance in statistics, in particular in relation with Wishart distributions, 
see \cite{CL, Fa, Ko, Mu, OR} well as the survey \cite{D}. 
For some applications in mathematical
 physics and representation theory, see for example  the survey \cite{Ne} and references therein. To our knowledge, beta distributions have so far only rarely been considered for indices
 beyond the critical value $\mu_0$. References in this case are \cite{U} and \cite{Sr}, where certain discrete indices are considered for which the associated beta measures become singular.

 Our aim in this paper is to study the extendability of the  Sonine formula \eqref{Sonine1} to larger ranges of the index $\nu$. This  will be achieved by analytic extension (with respect to $\nu$) of the beta probability measures as distributions,  and a detailed analysis when these distributions are still measures. 
 
 Our method is motivated by the theory of Gindikin for Riesz
distributions associated with symmetric cones  (see \cite{G1, G2}, Chapter 7 of \cite{FK}, and
the recent simplifications in \cite{S}).
Let us recall the basic facts, again for the case $V= \mathbb H_q(\mathbb F).$  
For indices $\alpha\in \mathbb C$ with $\Re \alpha > \mu_0 = n/q-1$, the Riesz probability distributions $R_\alpha$ on $V$ are defined by 
$$ R_\alpha(\varphi) = \frac{1}{\Gamma_\Omega(\alpha)}\int_{\Omega} \varphi(x)\Delta(x)^{\alpha-n/q}\,dx$$
where $\Gamma_\Omega$ is the gamma function associated with $\Omega$ and
$\Delta$ denotes the determinant on $V.$ 
According to the results by Gindikin, the measures $R_\alpha$ 
have a (weakly) analytic extension to distributions $R_\alpha\in \mathcal D^\prime (V)$ for all $\alpha \in \mathbb C$. This means that the mapping $\alpha \mapsto R_\alpha(\varphi) $ is analytic on $\mathbb C$ for each $\varphi \in \mathcal D(V).$ The distributions $R_\alpha$  are tempered and their support is contained in the closed cone $\overline \Omega.$
Moreover, $R_\alpha$ is a positive measure exactly if $\alpha $ belongs to the Wallach set
\begin{equation}\label{Wallach} \big\{0,\frac{d}{2},\ldots,(q-1)\frac{d}{2}= \mu_0\}\, \cup\, ]\mu_0,\infty[\,.\end{equation}
A simple proof for the necessity of this condition is given in \cite{S}.

We consider the beta measures 
$\beta_{\mu,\nu} $ in \eqref{betameasure} as compactly supported distributions on $V$ of order zero. Their extension to a larger range of the index $\nu$ is more involved than in the Riesz case. Indeed, the range of extension we are able to obtain depends on  $\Re \mu$.   To become precise, 
consider the open half planes
$$E_k:=\{\nu\in   \mathbb C:\> \Re\nu> \mu_0-k\}, \quad k\in \mathbb N_0$$
where $E_0\subset E_k \subset E_{k+1}$.
It is easily checked that for fixed $\mu\in E_0$, the mapping
\begin{equation}\label{dist0} E_0\,\to\,\cal D^{\prime}(V),
\quad \nu\mapsto 
\beta_{\mu,\nu}\end{equation}
is (weakly) analytic, i.e., $\nu\mapsto \beta_{\mu,\nu}(\varphi) $ is analytic for each test function $\varphi\in \mathcal D(V)$.
Recall that compactly supported distributions on $V$ extend 
continuously to $\mathcal E(V),$ the space $C^\infty(V)$ with its usual locally convex topology. 
In Theorem \ref{distribution-fortsetzung}, we prove:

\begin{theorem A}
For $k\in \mathbb N_0$ and $\mu\in \mathbb C$ with  $\Re\mu > \mu_0 +kq +1$, the mapping \eqref{dist0}
has a unique analytic extension from $E_0$ to $E_k$ with values in $\mathcal D^\prime(V)$. The distributions 
$\beta_{\mu, \nu}$ obtained in this way are compactly supported with support contained in $\overline{\Omega_e}$. 
Moreover, the Bessel functions 
 $\mathcal J_{\mu}$ and $\mathcal J_{\mu+ \nu}$ associated with $\Omega$ are  related by the  Sonine formula
\begin{equation}
\mathcal J_{\mu+\nu}(r)=\beta_{\mu,\nu} (\mathcal J_\mu^{r}) \quad \text{for all }\, r\in \Omega,\end{equation}
where $\mathcal J_\mu^r(x) = \mathcal J_\mu(\sqrt r x \sqrt r) \in \mathcal E(V).$ 
\end{theorem A}

We next  ask when the distributions $\beta_{\mu,\nu}$ are actually complex Radon measures or even probability measures.
The latter requires that $\mu, \nu\in \mathbb R.$  The following result is contained in Corollary \ref{summary1}:

\begin{theorem B}\label{main-intro}
 Let $\mathbb F=\mathbb R, \mathbb C$, $k\in \mathbb N$,  and 
$\mu\in \mathbb R$ with $\mu > \mu_0 + kq + 3/2.$  Then for $\nu\in E_k$, the following statements are
equivalent:
\begin{enumerate}
 \item[\rm{(1)}] $\nu$ is contained in the Wallach set \eqref{Wallach}.
\item[\rm{(2)}] The distribution $\beta_{\mu,\nu}$ is a  positive measure.
\item[\rm{(3)}] There exists a  probability measure $\beta\in M^1(\overline\Omega)$ such that 
$$ \mathcal J_{\mu + \nu}(r) = \int_{\overline\Omega} \mathcal J_\mu(\sqrt r s\sqrt r) d\beta(s) \quad \text{ for all } r\in \overline\Omega\,.$$
\end{enumerate}
In this case, the measure  $\beta\in M^1(\overline\Omega)$ in $(3)$ is unique, and  $\beta=\beta_{\mu, \nu}$. 
\end{theorem B}

We shall prove this result, as well as a counterpart for complex measures, actually in the more general setting 
of symmetric cones with Peirce constant $d=1$ or $2$.
This also includes the Lorentz cones in $\mathbb R\times \mathbb R^2$ and $\mathbb R\times \mathbb R^3.$ 
Without restriction on the Peirce constant $d$, our results (contained in Theorem \ref{not1}) are somewhat less
complete, but still give interesting restrictions on the indices which are necessary to assure that $\beta_{\mu,\nu}$ is a measure. This in particular concerns the case of 
quaternionic matrix cones.

We finally mention that the spherical polynomials and thus also the Bessel functions on Euclidean Jordan algebras depend only on the eigenvalues of
their argument. Considered
as functions of the spectra, the spherical polynomials can be identified with Jack polynomials whose index depends on $d$; this was first observed by Macdonald \cite{Ma1}.
There is a natural theory of hypergeometric expansions 
in terms of Jack polynomials (see \cite{Ka, Ma}) which encompasses
the theory on symmetric cones and is closely related with rational Dunkl theory, c.f. \cite{R2} and Lemma \ref{Besselestim} below.  Riesz distributions in this setting are Selberg densities, and their
analytic extension and consequences for integral representations of Bessel functions of Dunkl type will be studied in a forthcoming paper.

The organization of this  paper is as follows: The next section gives a short survey about Bessel functions on Euclidean Jordan algebras.
In Section 3 we discuss several facts concerning the  beta measures $\beta_{\mu,\nu}$ and the Sonine formula for $\Re \mu, \Re \nu > \mu_0 $.
Section 4 contains the main results of this paper on the analytic extension of the beta measures  and their consequences for Sonine integral representations of the Bessel functions.

\section{Bessel functions on Euclidean Jordan algebras}\label{sec2}

In this section we present
 some basic facts and notions on symmetric cones and   associated Bessel functions.
 We illustrate the general notions by the important example of matrix cones.
 For a background on symmetric cones and Jordan algebras we refer to the monograph \cite{FK}.  

A real algebra $V$ of finite dimension $n$ is called a (real) Jordan algebra if its multiplication $(x,y) \mapsto x\cdot y = xy\,$ satisfies 
$$  xy = yx \, \text{ and }\,  x(x^2y) = x^2(xy) \quad \text{for all } x,y \in V.$$
A real Jordan algebra is called Euclidean, if it has an identity $e\in V$ and a scalar product  $(\,.\,\vert \,.\,)$ such that $\, ( xy\vert z)  = (y\vert  xz)$ for all $x,y,z\in V.$ 
 It is called simple if it contains no non-trivial ideals. Let  $V$ be a Euclidean Jordan algebra.
Then the topological interior $\Omega$ of the set $\{x^2: x\in V\}$  is a symmetric cone. We recall that 
a symmetric cone 
$\Omega$ in a Euclidean vector space $V$ is  an open cone $\Omega\subseteq V$ which is proper (i.e. $\overline \Omega\cap -\overline\Omega = \{0\}$), 
 self-dual and homogeneous in the sense that the automorphism group
 of $\Omega$ acts in a transitive way. 
Let $G$ denote the identity component of this automorphism group and $K= G\cap O(V)$. Then already $G$ is transitive on $\Omega$, and there exist points $e\in \Omega$ 
such that  $K$ is the stabilizer of $e$ in $G$. Thus $\Omega\cong G/K$, which is a Riemannian symmetric space. With $e$ fixed as above,
there is a natural product in $V$ for which  $V$ becomes a 
Euclidean Jordan algebra with identity element $e$ and such that
$\overline\Omega = \{x^2: x\in V\}$ (See Thm. III.3.1.of \cite{FK}). 
 Every symmetric cone is a product of irreducible ones, and in the above way, the simple Euclidean Jordan algebras correspond to the 
irreducible symmetric cones. 

\begin{remarkex}

Let $\mathbb F$ be  one of the (skew) fields $\mathbb R, \mathbb C$
or  $\mathbb H$ with real dimension $d=1,2$ or $4$ respectively.
The usual conjugation in $\mathbb F$  is denoted by $t\mapsto \overline t$ and the real part of $t\in \mathbb F$ by $\mathfrak R t = \frac{1}{2}(t + \overline t).$
Let
\[H_q(\mathbb F) := \{x\in M_{q}(\mathbb F): x= x^*\}\]
be the space of Hermitian $q\times q$-matrices over $\mathbb F$, where $\,x^* =
\overline x^t.$ 
We consider $H_q(\mathbb F)$ as a  Euclidean vector space with  scalar product
 $\langle x, y\rangle = \mathfrak R \,\text{Tr}(xy)$,   where  $\,\text{Tr}(x) = \sum_{i=1}^q x_{ii}$ denotes the usual trace. With this scalar product and the Jordan product $x\cdot y= \frac{1}{2}(xy+yx)$, the space $\mathbb H_q(\mathbb F)$ becomes a simple Euclidean Jordan algebra
 with identity $e= I_q$.  
 The associated symmetric cone is given by 
 $$ \Omega_q(\mathbb F) =  \{x\in H_q(\mathbb F): 
 x \text{  positive definite}\}.$$ 
The pairs $(G,K)$ are in this case  $(GL_q^+(\mathbb R), SO_q(\mathbb R)),$ $(GL_q(\mathbb C), U_q(\mathbb C))$ and $(GL_q(\mathbb H),$ $ U_q(\mathbb H)),$ respectively, where the action of
$G$  on $\Omega_q(\mathbb F)$ is given by $r \mapsto grg^*$. Notice that this reduces to conjugation when restricted to $K$. 

\end{remarkex}

Let now $V$ be a simple Euclidean Jordan algebra and $\Omega$ the associated symmetric cone. 
A Jordan frame in $V$ is a complete set $c_1, \ldots, c_q\in V$ of orthogonal primitive idempotents, i.e.
$$ c_i^2 = c_i, \, c_ic_j = 0 \text{ if } i\not= j, \, c_1 + \ldots c_q = e.$$ 
The group $K$ acts transitively on the set of Jordan frames, and 
their common cardinality $q$ is called the rank of $V$ (or $\Omega$).
The rank of $V$ is related to its real dimension $n$  via
$$ n= q + \frac{d}{2}q(q-1),$$
where $d$ is the so-called Peirce constant, see p.71 of \cite{FK}. 
Each $x\in V$ admits a decomposition $x= k\sum_{i=1}^n \xi_i c_i$
with $k\in K$ and unique real numbers $\xi_1\geq \ldots \geq
\xi_q$ which are called the eigenvalues of $x$ (Sect. VI.2. of \cite{FK}).
The Jordan trace and determinant of $x$ are defined by 
$$\text{tr}(x) = \sum\xi_i, \quad \Delta(x) = \prod \xi_i\,.$$
Both functions are $K$-invariant.
 
 \begin{remarkex}
In the Jordan algebras $H_q(\b F)$, a natural Jordan frame consists
of the matrices $c_i =E_{ii}, \, 1\leq i \leq q$ (having entry $1$ in position $(i,i)$ and $0$ else). The eigenvalues of $x\in \mathbb H_q(\mathbb F)$ are  the usual (right) eigenvalues, and
$\Delta$ coincides with
the usual determinant 
 if $\b F=\b R$ or $\b C$, while for
$\b F=\b H$ it is given by the so-called Moore determinant, see \cite{Asl}. 
\end{remarkex}

The simple Euclidean Jordan algebras are classified. 
Up to isomorphism, there are the series $H_q(\mathbb F)$  with $\mathbb F= \mathbb R, \mathbb C, \mathbb H$, the exceptional Jordan algebra $H_3(\mathbb O)$ over the octonions,  as well as the Jordan algebras
$V=\b R\times \b R^{q-1}, \, q\geq 3,$ with Jordan product
$\,(\lambda,u)\cdot (\mu,v) = (\lambda\mu + \langle  u, v\rangle, \lambda v + \mu u)$, where 
$\langle\,.\,,\,.\,\rangle$ denotes the usual Euclidean 
scalar product on $\b R^{q-1}$. In this case, $\Omega$ is the Lorentz cone
\[\Lambda_q = \{ (\lambda,u) \in \b R\times\b R^{q-1}: 
\lambda^2- \langle u, u\rangle >0, \lambda >0\}.\]
The following table summarizes these Jordan algebras and their structure data.
\medskip

\begin{center}
\renewcommand{\arraystretch}{1.4}
\begin{tabular}{|c|c|c|c|c|}\hline
 $V$  &  $\Omega$ & rank & $d$ & $n=\text{dim V} $  \\ \hline\hline
 $H_q(\mathbb R)$   & $ \Omega_q(\mathbb R)$   &$q$ & $ 1$ & $\frac{1}{2}q(q+1) $       \\ \hline  
$H_q(\mathbb C)$   & $ \Omega_q(\mathbb C)$   &$q$ & $ 2$ & $q^2$       \\ \hline 
$H_q(\mathbb H)$   & $ \Omega_q(\mathbb H)$   &$q$ & $ 4$ & $q(2q-1)$       \\ \hline 
$H_3(\mathbb O)$ & $\Omega_3(\mathbb O)$ & $3$ & $8$ & $27$ \\ \hline
$ \mathbb R\times \mathbb R^{q-1}$ & $\Lambda_q$ & $2$ & $q-2$ &$q$ \\ \hline
                        
\end{tabular}
\end{center}

\bigskip

In this paper, we shall always assume that $V$ is a simple Euclidean Jordan algebra with associated symmetric cone $\Omega$   and that the scalar product of $V$ is given by $$ \langle x,y \rangle = \text{tr}(xy),$$ where $xy$ denotes the Jordan product. (This is no loss of generality, c.f. Section III.4 of \cite{FK}.)
We  need some further notation:
On $V$ we use the  partial orderings
$$ x<y \, :\, \Longleftrightarrow \, y-x\in \Omega \quad \text{and }\, x\leq y :\, \Longleftrightarrow \, y-x\in \overline\Omega.$$
The quadratic representation $P$ of $V$ is defined by 
$$ P(x):= 2 L(x)^2 - L(x^2), \,\, x\in V,$$
where $L(x)\in End(V)$ denotes the left multiplication by $x$ on $V$, i.e. $L(x)y = xy.$ 
For the Jordan algebras $H_q(\mathbb F)$, the quadratic representation is given by 
$$ P(x)y = xyx$$
where on the right side, the product is the usual matrix product (\cite{FK}, Section II.3). An element $x\in V$ is invertible in $V$ if and only if $P(x)$ is invertible, and in this case $P(x^{-1}) = P(x^{-1}).$ 
We finally mention an important invariance property:
Let $r,s \in \Omega$. Then by Lemma XIV.1.2 of \cite{FK},  there exists $k\in K$ such that
\begin{equation}\label{quad} P(\sqrt r\,)s = kP(\sqrt s\,)r.\end{equation}

For normalizations we need the gamma and beta function associated with the cone $\Omega$ (\cite{FK}, Chapter VII.1). They are defined by
\begin{align*}
 \Gamma_\Omega(z) &= \int_{\Omega}e^{-\text{tr}(x)} \Delta(x)^{z-n/q}\> dx, \\
 B_\Omega(z,w) & = \int_{\Omega_e }  \Delta(x)^{z-n/q}\Delta(e-x)^{w-n/q} dx,
\end{align*}
where $dx $ is the Lebesgue measure on $V$ induced by the scalar product $\langle\, .\,,\,.\,\rangle$ and 
\[\Omega_e = \{x\in \Omega: x< e\}\]
Both integrals are absolutely convergent for all $z,w \in \mathbb C$ with $\Re z, \Re w > \mu_0,$
where 
\begin{equation}\label{mu0} \mu_0 := \frac{n}{q}-1 = \frac{d}{2}(q-1).\end{equation}
By Corollary VII.1.3 of \cite{FK}, $\Gamma_\Omega$ can be expressed in terms of the classical gamma function as  
\begin{equation}\label{gamma-mult} \Gamma_\Omega(z) = (2\pi)^{(n-q)/2}
\prod_{j=1}^q\Gamma\big(z-\frac{d}{2}(j-1)\big).\end{equation}
Moreover, 
$$ B_\Omega(z,w) = \frac{\Gamma_\Omega(z)\Gamma_\Omega(w)}{\Gamma_\Omega(z+w)},$$
see Theorem VII.1.7 of \cite{FK}.
Notice that $\Gamma_\Omega$ is meromorphic on $\mathbb C$ without zeros, and its set of poles is
\[ \bigl\{0, \frac{d}{2}, \ldots, (q-1)\frac{d}{2}= \mu_0\bigr\} \, - \mathbb N_0\,.\]

The basic functions for the harmonic analysis on the cone $\Omega$ and building blocks for related special functions are the
 so-called spherical polynomials. For their definition, 
 recall that a $q$-tuple $(\lambda_1, \ldots, \lambda_q) \in \mathbb N_0^q$ is called a partition if $\lambda_1 \geq \ldots \geq \lambda_q$. 
We shall write $\lambda\geq 0$ for short.  
 The spherical polynomials associated with $\Omega$ are indexed by partitions $\lambda = (\lambda_1, \ldots, \lambda_q)   \in \b N_0^q$ 
 and are defined by
\[ \Phi_\lambda(x) = \int_K \Delta_\lambda(kx) dk, \quad x\in V\]
where $dk$ is the normalized Haar measure on $K$ and $\Delta_\lambda$ is the generalized
power function on $V$,
\[ \Delta_\lambda(x) = \Delta_1(x)^{\lambda_1-\lambda_2}
\Delta_2(x)^{\lambda_2-\lambda_3} \cdot\ldots\cdot \Delta_q(x)^{\lambda_q}.\]
Here the $\Delta_i(x)$ are the principal minors of 
$\Delta(x)$, see Section VII.1 of \cite{FK} for details. For the 
matrix algebras $V=H_q(\mathbb F)$, the $\Delta_i(x)$ are just the usual principal minors.  The power function
$\Delta_\lambda$ is a homogeneous polynomial of degree $\, |\lambda| =
\lambda_1 + \ldots + \lambda_q $ which is  positive on $\Omega$.
For convenience, we shall work with renormalized spherical polynomials 
$Z_\lambda = c_\lambda \Phi_\lambda, $
where the constants $c_\lambda >0$ are 
such that
\begin{equation}\label{power-tr}
(\text{tr}\,x)^k \,=\, \sum_{|\lambda|=k} Z_\lambda(x)
\quad\quad\text{for}\quad
 k\in \mathbb N_0\,,
\end{equation}
see Section XI.5.~of \cite {FK}.
The $Z_\lambda$ are $K$-invariant and thus
 depend only on the eigenvalues of their argument. 
 In view of \eqref{quad},  $$
Z_\lambda\bigl (P(\sqrt r\,)s) = Z_\lambda\bigl (P(\sqrt s\,)r) \quad \text{ for all } r,s \in \Omega.$$

We mention that for each symmetric cone the associated spherical polynomials are
given in terms of 
Jack polynomials (\cite{St}).   More precisely, it was observed by Macdonald in \cite{Ma1} that for $x\in V$ with eigenvalues $\xi = (\xi_1, \ldots, \xi_q)\in
\mathbb R^q$,
\begin{equation}\label{identjack}Z_\lambda(x) =
C_\lambda^\alpha(\xi) \quad \text{with}\quad \alpha = \frac{2}{d}\end{equation}
where the $C_\lambda^\alpha$  are the  Jack polynomials of $q$ variables and index $\alpha>0$, normalized such that 
\begin{equation}\label{power-tr2} (\xi_1 + \ldots + \xi_q)^k = \sum_{ |\lambda|=k } C_\lambda^\alpha(\xi).\end{equation}
The $C_\lambda^\alpha$  are homogeneous of degree $|\lambda|$ and symmetric in
their arguments. 

For a simple Euclidean Jordan algebra with Peirce constant $d$,  the associated ($\mathcal J$-) Bessel functions  
are defined on $V^\mathbb C$ as 
\begin{equation}\label{power-j}
 \mathcal J_\mu(z) =
 \sum_{\lambda\geq 0} \frac{(-1)^{|\lambda|}}{(\mu)_\lambda \,|\lambda|!}
Z_\lambda(z),
\end{equation}
with the generalized Pochhammer symbol  
\[
(\mu)_\lambda := \,\prod_{j=1}^q
\bigl(\mu-\frac{d}{2}(j-1)\bigr)_{\lambda_j},\]c.f. \cite{FK}. Here it is assumed that $\mu \in \mathbb C$ with $(\mu)_\lambda \not=0$ for
all $\lambda \geq 0$, which is for example satisfied if $\Re \mu > \mu_0$.  The function $\mathcal J_\mu$ is analytic on $V^\mathbb C.$

In  the rank one case $q=1,$ we have
 $\Omega= ]0,\infty[$, $\,\mu_0 = 0,$  and the Bessel function $ \mathcal J_\mu$ is
independent of $d$ and satisfies
\[ \mathcal J_\mu\bigl(\frac{z^2}{4}\bigr) = j_{\mu-1}(z) \]
in the notation of the introduction.

For the matrix cones $\Omega_q(\mathbb F)$ it is well-known 
 (see e.g. \cite{R2} or \cite{RV3} for details) that 
for certain indices $\mu$, the associated Bessel functions $\mathcal J_\mu$ occur 
as spherical functions of flat symmetric spaces. In fact, fix some integer $p\geq q$ and denote by  $M_{p,q}= M_{p,q}(\mathbb F)$ the space of $p\times q$-matrices over $\mathbb F$. 
Consider the Gelfand pair 
$(M_{p,q}\rtimes U_p, U_p),$  where the unitary group $U_p:=U_p(\mathbb F)$ acts on $M_{p,q}$ by left multiplication. The double coset space
$M_{p,q}  \rtimes U_p//U_p$ 
may be naturally identified with the orbit space 
$M_{p,q}^{U_p}\,$ which is in turn homeomorphic with the closed cone 
$$ \overline{\Omega_q(\mathbb F)} = \{ x\in H_q(\mathbb F): x \text{ positive semidefinite}\}$$ 
via the mapping
$$ U_p.x \mapsto \sqrt{x^*x}\,.$$
Considered as functions on the cone $\overline{\Omega_q(\mathbb F)}$, the bounded spherical functions
of the Gelfand pair $(M_{p,q}\rtimes U_p, U_p)$ are precisely given by the Bessel functions 
\[\varphi_s^\mu(r) = \mathcal J_\mu\bigl(\frac{1}{4}rs^2r), \>\> s \in \overline{\Omega_q(\mathbb F)}\quad \text{with } \mu= pd/2.\]

 As spherical functions, these Bessel functions  have an integral representation of Harish-Chandra type. 
 Analytic continuation with respect to $\mu$ leads to the following integral representation, see 
 \cite{R2}, Section 3.3 as well as \cite{H} for $\mathbb F = \mathbb R$.

\begin{proposition}\label{int-rep-mat-bessel}
Let $d= \text{dim}_\mathbb R \mathbb F$ and  $\mu\in\mathbb C$ with 
$\Re\mu>d(q-1/2).$
 Then for all $x\in H_q(\mathbb F)$, the Bessel function of index $\mu$ associated with $H_q(\mathbb F)$ satisfies
\begin{equation}\label{int-rep-exp-matrix}
\mathcal J_\mu(x^2)=\frac{1}{\kappa_\mu}\int_{B_q(\mathbb F)} e^{-2i\langle v,
  x\rangle}\Delta(I-v^*v)^{\mu-1-d(q-1/2)}dv
\end{equation}
with 
$B_q(\mathbb F) =\{v\in M_q(\mathbb F):\> v^*v<I_q\}$, the scalar product $\langle x,y\rangle = \Re \,Tr (x^*y)$ on $M_q(\mathbb F)$  and   $$\kappa_\mu= \int_{B_q(\mathbb F)}  \Delta(I-v^*v)^{\mu-1-d(q-1/2)}dv.$$
\end{proposition}

Formula \eqref{int-rep-exp-matrix} generalizes the Laplace representation 
\eqref{intrep-1-dim-allg-exp} to  higher rank.

\section{Beta measures and the Sonine formula on symmetric cones}

 Throughout this paper,  $M_b(X)$ is the set of bounded, regular, complex Borel measures on a locally compact Hausdorff space $X$
 and $M^1(X)$ the set of all probability measures in  $M_b(X)$.

As before, let $V$ be a simple Euclidean Jordan algebra and $\Omega$ the associated symmetric cone.
For  $ \mu, \nu \in \mathbb C$ with $\Re\mu, \Re\nu > \mu_0 = \frac{n}{q}-1\,$ 
we introduce the beta measures
\begin{equation}\label{def-beta-dist}
d\beta_{\mu,\nu}(x):= \frac{1}{ B_\Omega(\mu,\nu)}  \Delta(x)^{\mu-n/q}\Delta(e-x)^{\nu-n/q}
\cdot 1_{\Omega_e}(x) dx \in M_b(\Omega)\end{equation}
which we also consider as measures on $V$ (or on $\overline\Omega$) with compact support 
$\overline{\Omega_e}.$
The $\beta_{\mu, \nu}$ are probability measures for  $\mu,\nu\in]\mu_0,\infty[$. 
We here do not use the  notion ``beta distributions'',
 as we shall  study (tempered) distributions below and want
 to avoid any misunderstanding. 

 Our starting point 
 is the following  Sonine formula \eqref {intrep-1-dim-allg} for  Bessel functions on Euclidean Jordan algebras, which generalizes formula \eqref{Sonine1} announced in the introduction. For   $\mathbb H_q(\mathbb R)$ it goes already back to \cite{H} (formula (2.6')).

\begin{theorem}\label{int-rep-classical} Let $V$ be a simple Euclidean Jordan algebra. Then
for all $\mu,\nu\in\mathbb C$ with $\Re\mu,\Re\nu >\mu_0$  and $x \in V$,
$$\mathcal J_{\mu+\nu}(x)=
\int_{ \Omega_e } \mathcal J_\mu\bigl(P(\sqrt r\,)x\bigr) d\beta_{\mu,\nu}(r).$$

\end{theorem}

\begin{proof} For indices $\alpha, \beta_1, \beta_2\in \mathbb C$ with $\Re\alpha, \Re\beta_i >\mu_0$ 
 consider the hypergeometric function 
$$\, _1F_2(\alpha; \beta_1,\beta_2;z):= 
\sum_{\lambda\geq 0} \frac{(\alpha)_\lambda}{(\beta_1)_\lambda(\beta_2)_\lambda\, |\lambda|!} \,Z_\lambda(z)$$
which is holomorphic on $V^\mathbb C$ (Proposition XV.1.1. of \cite{FK}).
Consider first $x\in \Omega_e$. As $G$ is transitive on $\Omega$, there exists $g\in G$ with $x=ge.$ According to Proposition XV.1.4 of \cite{FK}, 
$$ \, _1F_2(\mu;\beta, \mu+\nu;ge) = \frac{1}{B_\Omega(\mu,\nu)}\int_{\Omega_e} \,_0F_1(\mu; g r)d\beta_{\mu,\nu}(r).$$
With $\beta = \mu$, this becomes
$$ \mathcal J_{\mu+\nu}(-x) =\, _0F_1(\mu+\nu;ge) =  \frac{1}{B_\Omega(\mu,\nu)}
\int_{\Omega_e}\mathcal J_\mu(-gr)d\beta_{\mu,\nu}(r).$$

 By Theorem III.5.1 of \cite{FK},  $g$ can be written in polar form as 
$g=P(s)k$ with $s\in \Omega, k\in K$. Thus $x = ge = P(s)e = s^2$ and $g= P(\sqrt x\,)k.$ The measures   $\beta_{\mu,\nu}$ and the function $\mathcal J_\mu$ are $K$-invariant. Thus in view of  \eqref{quad}, 
$$ \mathcal J_{\mu+\nu}(-x) = c\int_{\Omega_e}\mathcal J_\mu\bigl(-P(\sqrt{x})r\bigr) d\beta_{\mu,\nu}(r) =  c\int_{\Omega_e}\mathcal J_\mu\bigl(-P(\sqrt{r})x\bigr) d\beta_{\mu,\nu}(r) $$
with $c = B_\Omega(\mu,\nu)^{-1}.$ The last formula extends analytically to all $x\in V.$

\end{proof}


We conclude with some remarks concerning the matrix cones $\Omega_q(\mathbb F).$ 

\begin{remark}
(1)
It follows form the analysis in Section 3 of \cite{R2} that for $\mu >\mu_0$ and $\nu = \mu_0,$ there exist
degenerated  beta probability measures $\beta_{\mu,\nu}$ on $\Omega_q(\mathbb F)_e$ 
such that the mapping
$\nu\mapsto\beta_{\mu,\nu}$ becomes weakly continuous on $[\mu_0, \infty[.$ In this way Theorem \ref{int-rep-classical}
 extends to indices $\nu\ge\mu_0$ and $\mu>\mu_0$. In  \cite{U} and  \cite{Sr} some singular  beta
measures are studied for $\mathbb F=\mathbb R$.

\smallskip

(2)
 Formula \eqref{int-rep-exp-matrix} 
may be regarded as a special case of Theorem \ref{int-rep-classical}
with the parameters $(qd/2, \mu-qd/2)$ instead of $(\mu, \nu).$ To check this, 
we first recall from formula (3.4) of \cite{R2} that for $x\in H_q(\mathbb F),$ 
\begin{equation}\label{exp-darstellung2}
\mathcal J_{qd/2}(x^*x)=\int_{U_q} e^{-2i\langle u, x\rangle}  \, du 
\end{equation}
where $du$ is the normalized Haar measure on $U_q=U_q(\mathbb F)$ and the scalar product is that of Proposition \ref{int-rep-mat-bessel}. We also need the integral formula for the polar decomposition of $M_q(\mathbb F)$ (see \cite{FT} or Section 3.1 of \cite{R2}):
\begin{equation*}
\int_{M_q(\mathbb F)}f(x)dx \, = \,C\! \int_{U_q}\int_{\Omega_q(\mathbb F)} f(u\sqrt r)\>
\Delta(r)^{qd/2-n/q} \, dr du
\end{equation*}
with some constant $C=C_q>0.$ 
Let $\Re \mu > d(q-1/2).$ 
Then identity \eqref{int-rep-exp-matrix} becomes

\begin{align}
\mathcal J_{\mu}(x^*x)&=\, C\! \int_{B_q(\mathbb F)} e^{-2i\langle v,
  x\rangle}\Delta(I-v^*v)^{\mu-1-d(q-1/2)}dv \notag\\
&= C\! \int_{\Omega_q(\mathbb F)_e}\Bigl( \int_{U_q} e^{-2i\langle u\sqrt r,x\rangle } du\Bigr)
\Delta(r)^{qd/2-n/q} \Delta(I-r)^{\mu-1-d(q-1/2)}\, dr\notag\\
 &= C\! \int_{\Omega_q(\mathbb F)_e} \mathcal J_{qd/2}(\sqrt r\, x^*x\sqrt r)\,
\Delta(r)^{qd/2-n/q} \Delta(I-r)^{\mu-1-d(q-1/2)} dr.\notag
\end{align}
Put $\nu := \mu - qd/2$ and notice that $\Re\nu > \mu_0$.
In view of  the normalization $\mathcal J_\mu(0)=1$ this is equivalent to
\begin{equation*}
\mathcal J_{\nu+qd/2}(s)=\int_{\Omega_q(\mathbb F)_e} \mathcal J_{qd/2}(\sqrt r\, s\sqrt r)\, d\beta_{qd/2,\nu}(r)
\end{equation*}
for all $s\in \overline{\Omega_q(\mathbb F)},$ which is a special case
of Theorem \ref{int-rep-classical} as stated.

\end{remark}

In the next section we shall construct an extension of 
 Theorem \ref{int-rep-classical} with respect to the ranges  of the indices  $\mu, 
 \,\nu$. 
Before that, we mention a special case in the matrix cone setting which follows from group theory:

\begin{proposition}\label{group-case}
Let 
 $\mu=pd/2$ and $\nu=\widetilde pd/2$ with integers $p\ge q$ and $\widetilde
p\ge0$. Then there exists a unique probability measure 
$\widetilde\beta_{\mu,\nu}$ on on the matrix cone $\overline{\Omega_q(\mathbb F)}$ such that for all
 $r\in\overline{\Omega_q(\mathbb F)}$,
$$\mathcal J_{\mu+\nu}(r)=
\int_{\overline{ \Omega_q(\mathbb F)}} \mathcal J_\mu\bigl(\sqrt{r}s\sqrt{r}\bigr) \, d\widetilde\beta_{\mu,\nu}(s).$$
\end{proposition}

\begin{proof} For brevity we omit $\mathbb F$ in the notation of the relevant matrix spaces. 
Recall that the functions 
\[\varphi_s^\mu(r) = \mathcal J_{\mu}\bigl(\frac{1}{4}rs^2r\bigr) \quad s,r \in \overline{\Omega_q}\]
can be naturally identified with the bounded spherical functions of the Gelfand pair $(M_{p,q}  \rtimes U_{p},U_{p}),$ and the
$\varphi_s^{\mu + \nu}$ with those of the pair 
$(M_{p+\widetilde p,q}  \rtimes U_{p+\widetilde p},U_{p+\widetilde p}).$ In Section 6 of \cite{RV3} it is deduced from this characterization by a positive definiteness argument that the functions $\phi_s^{\mu + \nu}$ have a representation 
\[ \varphi_s^{\mu+\nu}(r) = \int_{\overline{\Omega_q}} \,\varphi_t^\mu(r)d\alpha_{\mu, \nu;s}(t),\]
with a unique probability measure $\alpha_{\mu, \nu;s}$, see formula (6.1) of \cite{RV3}. (There is a misprint in \cite{RV3}: the indices $p_1, p_2$ of the groups $G_{p_1}, G_{p_2}$ are mixed up). 
With $s=2I_q$ this immediateley yields our claim.
\end{proof}

\section{Beta distributions and extension of the Sonine formula}

In this section we present 
 an analytic extension of  Theorem \ref{int-rep-classical} with respect to the parameters $\mu, \nu$  by 
 distributional methods. Let us first fix some notation.
 For an open subset  $U$  of some finite dimensional vector space $V$ over
$\mathbb R$, denote by $\cal D(U)$ the space of compactly supported $C^\infty$-functions
on $U$ and by  $\cal D^\prime(U)$
 the space of distributions on $U$. We write 
 $\cal E(U)$ for the space $C^\infty(U)$ with its usual
 Fr\'echet space topology; its dual 
 $\cal E^\prime(U)$ coincides with the space of compactly supported distributions on $U$. Further, 
 $\cal D^{\prime k}(U)$ denotes the space  of distributions of order $\le k$. Recall  that  $\cal D^{\prime 0}(U)$ consists of those distributions
which are given by a  (not necessarily bounded) complex  Radon measure  
on $U$. In particular,
 each locally integrable function 
$f\in L^1_{loc}(U)$ defines a regular distribution 
$T_f\in \cal D^{\prime 0}(U)$ via $T_f(\varphi)=\int_U \varphi(x)f(x)\> dx.$

We shall also consider regular distributions associated with functions  $f_\lambda\in L^1_{loc}(U)$ 
where $f_\lambda$ depends
analytically on some parameter $\lambda \in D$ with some open, connected set $D\subseteq \mathbb C$. 
We ask for which parameters the associated distributions
 $T_{\lambda}:=T_{f_\lambda}\in\cal D^{\prime 0}(U)$ admit
extensions to  distributions of order $0$ on $V$.  
We  shall need the following  observation
 from Sokal \cite{S}; see Lemmata 2.1, 2.2 and
Proposition 2.3 there.

\begin{lemma}\label{Sokal}
Let $U\subseteq V$ and $D\subseteq \mathbb C$ be as above, and let 
$$F:U\times D\to\mathbb C, \quad (x,\lambda)\to f_\lambda(x):=F(x,\lambda)$$
be a continuous function such that $F(x,.)$ is analytic in $D$ for
 each $x\in U$. Define $T_\lambda \in \cal D^{\prime 0}(U)$ by
 \[ T_\lambda(\varphi) = \int_U \varphi(x) f_\lambda(x) dx.\]
 Then the following hold:
\begin{enumerate}\itemsep=-1pt
 \item[\rm{(1)}] The map $\,\lambda\mapsto T_\lambda\,, \,\, D \to 
 \cal D^\prime(U)  $ is (weakly) analytic in the sense that  
$\lambda\mapsto T_\lambda(\varphi)$
is analytic for all $\varphi \in\cal D(U)$.

\item[\rm{(2)}]  Let  $D_0\subseteq D$ be a nonempty open set, and let
  $\lambda\mapsto \widetilde T_\lambda, \,\, D \to \cal D^\prime(V)$  
   be an analytic map such that $\widetilde T_\lambda$ extends $T_\lambda$
for each $\lambda\in D_0$. Then  $\widetilde T_\lambda$ extends $T_\lambda$
for each $\lambda\in D$. Moreover, for each $\lambda\in D$ with 
$\widetilde T_\lambda\in \cal D^{\prime0}(V)$ one has $f_\lambda\in
L^1_{loc}(\,\overline U\,)$, that is $f_\lambda$ is integrable over each
sufficiently small neighborhood in $V$ of any point $x\in\overline U$.
In particular, if $\,\overline U$ is compact, then  $f_\lambda$ is the density of
a bounded measure.
 \end{enumerate}
\end{lemma}

We start our considerations on symmetric cones with an injectivity result 
which is of interest in its own and will be of importance in the sequel.
Let again $\Omega$ be an irreducible symmetric cone and $V$ the associated simple Euclidean Jordan algebra. 
Following \cite{Di}, we consider the Schwartz space of the closed cone $ \overline \Omega$, 
\[ \mathcal S(\overline\Omega) := \big\{ f\in C^\infty(\overline\Omega): \|f\|_{\alpha, \beta, \overline\Omega} := \|x^\beta\partial^\alpha f\|_{\infty, \overline\Omega}\, < 
\infty \,
\text{ for all $\alpha, \beta \in \mathbb N_0^q$}\big\}.
\]
 Here  $C^\infty(\overline\Omega)$ denotes the space of continuous functions on $\overline\Omega$ which are smooth on $\Omega$ and whose
 partial derivatives extend continuously to  $\overline\Omega$. We note that each $f\in C^\infty(\overline\Omega)$ 
can be extended to a smooth function on $V$. This follows by the Whitney extension theorem 
(see \cite{Bie}, Theorem 2.6 and Propos. 2.16), because $\overline\Omega$ is a 
semi-algebraic (and hence subanalytic) subset of $V$ with dense interior. Therefore
\[ \mathcal S(\overline\Omega) = \big\{ f\big\vert_{\overline\Omega}: \,f\in C^\infty(V), 
\, \|f\|_{\alpha, \beta, \overline\Omega} < \infty \,
\text{ for all $\alpha, \beta \in \mathbb N_0^q$}\big\}.\]
The same approximation argument as for the density of $\mathcal D(\mathbb R^n)$ in $\mathcal S(\mathbb R^n)$ (see for instance \cite{Ru}) shows that the space
\[\mathcal D(\overline\Omega):= \{ f\big\vert_{\overline\Omega}: f\in \mathcal D(V)\}\]
is dense in $\mathcal S(\overline\Omega)$ with respect to the seminorms $\|.\|_{\alpha,\beta, \overline\Omega}.$
We  denote by $\mathcal S^\prime(\overline\Omega)$ the dual of the locally convex space $\mathcal S(\overline\Omega),$ i.e. 
the space of tempered distributions on $\overline\Omega$. Let $\,\Re\mu > d(q-1) + 1 = 2\mu_0+1.$
Then according to Theorem 2.2. of \cite{Di}, 
the Hankel transform 
\[ f\mapsto \widehat f^{\,\mu}, \quad \widehat f^{\,\mu}(r) = \int_{\Omega} f(s) \mathcal J_\mu\bigl(P(\sqrt s\,)r\bigr)\Delta(s)^{\mu-n/q} ds\]
is a homeomorphism of $\mathcal S(\overline\Omega)$. Actually, this is stated in \cite{Di} for 
$\Re\mu>\mu_0$, but the proof requires absolute convergence of the inverse Laplace integral representing the Bessel function, which is guaranteed only for $\,\Re\mu > d(q-1) + 1,$ see \cite{FK}, Proposition XV.2.2.
The stated homeomorphism property  allows to deduce  the following injectivity result.

 \begin{theorem}\label{injectivity}
 Let $\mu \in \mathbb C$ with $\Re \mu > 2\mu_0+1$. For $r\in \overline\Omega\,$ define 
 $$\mathcal J_\mu^{r}(x):= \mathcal J_\mu\bigl(P(\sqrt r\,)x\bigr)\in \cal E(V).$$
 Suppose that $T\in \mathcal E^\prime (V)$ has compact support which is contained in
$\overline\Omega$. Then the following hold:
\begin{enumerate}
\item[\rm{(1)}] If $T(\mathcal J_\mu^{r}) = 0$ for all $r \in \overline\Omega$, then $T=0.$
\item[\rm{(2)}] Suppose that $\mathcal J_\mu^{r}$ is bounded for each $r\in \overline{\Omega}$, and that there is a bounded measure $\beta\in M_b(\overline\Omega)$ 
(also considered as a measure on $V$)  such that
$$ T(\mathcal J_\mu^{r}) = \int_{\overline\Omega} \mathcal J_\mu^{r}(s) \,d\beta(s) \quad \text{for all }\, r\in \overline\Omega\,.$$
Then $T = \beta$. 
\end{enumerate}

\end{theorem}

\begin{proof}
We first observe that $T$ belongs to $\mathcal S^\prime(\overline\Omega)$. Indeed,
choose a compact, convex and connected subset $K\subset \overline\Omega$ containing the support of $T$,
and let $k$ denote the order of $T$. Then according to Theorem 2.3.10 of \cite{Ho}, there exists a constant $C>0$ such that for all 
$\varphi \in \mathcal E(V),$
\begin{equation}\label{Horabsch} |T(\varphi) | \leq \, C \sum_{|\alpha|\leq k} 
\|\partial^\alpha \varphi \|_{\infty, K} \,.\end{equation}
This shows that $T\in \mathcal S^\prime(\overline\Omega)$ and that the inclusion 
$$\{ T\in \mathcal E^\prime(V): \text{supp}\, T\subset \overline\Omega\} \to \mathcal S^\prime(\overline\Omega)$$
 is injective.
Now let $\varphi \in \mathcal D(\overline\Omega).$ 
It is easy to check that the mapping $\,r\mapsto \mathcal J_\mu^{\,r}\,, \,\overline\Omega\to \mathcal E(V)\,$ is continuous. Therefore
\[ \int_{\text{supp} \varphi} \mathcal J_\mu^{r}\,\varphi(r)\Delta(r)^{\mu-n/q}dr\]
is well-defined as an integral with values in $\mathcal E(V)$ (see e.g. Section 3 of \cite{Ru}), and we obtain
\begin{equation}\label{distidentity}  T(\widehat \varphi^{\, \mu}) = 
T\bigl(\int_{\text{supp} \varphi}  \mathcal J_\mu^{r}\,\varphi(r)\Delta(r)^{\mu-n/q}dr\bigr) = 
 \int_{\text{supp} \varphi}  T(\mathcal J_\mu^{r})\,\varphi(r)\Delta(r)^{\mu-n/q}dr . \end{equation}
 In the situation of part (1), it follows that $T(\widehat \varphi^{\, \mu}) = 0$. 
 As $\mathcal D(\overline\Omega)$ is dense in $\mathcal S(\overline\Omega)$ and the Hankel transform is a homeomorphism of $\mathcal S(\overline\Omega)$, this implies that $T=0$ as an element of $\mathcal S^\prime(\overline\Omega),$ 
which yields assertion (1).
In the situation of part (2), identity \eqref{distidentity} 
leads to
\[T(\widehat \varphi^{\, \mu})  = \int_{\overline\Omega} \,\widehat \varphi^{\,\mu} (s)\,d\beta(s),\] 
and the same argument as above shows that $T = \beta.$  \end{proof}

The following estimate implies that already for $\Re \mu\geq \mu_0 +\frac{1}{2}$, the Bessel functions $\mathcal J_\mu^{\,r}$ with $r\in \overline\Omega$ 
are indeed bounded on $\Omega$ as required in part (2) of the above theorem.

\begin{lemma}\label{Besselestim} Let $\Re \mu\geq \mu_0 +1/2$. Then 
$$ |\mathcal J_\mu(x)|\leq \sqrt{2^q q!} \quad \text{for all } x\in V.$$

\end{lemma}

For further bounds on $\mathcal J$-Bessel functions see \cite{Mo} and \cite{Na}; they do however not cover Lemma \ref{Besselestim} above.  
Our proof of this Lemma will be based on the connection between
$\mathcal J_\mu$ and Bessel functions of Dunkl type associated with the root system
$$B_q = \{\pm e_i, 1\leq i \leq q\}\cup \{ \pm e_i\pm e_j: 1\leq i <j \leq q\}\subset \mathbb R^q$$ as established in \cite{R2}. 
The reflection group associated with $B_q$ is the hyperoctahedral group $G= S_q \ltimes \mathbb Z_2^q$. For a general background on Dunkl theory
see \cite{dJ}, \cite{R1} and the references cited there.
Let $E_k^B: \mathbb C^q\times \mathbb C^q\to \mathbb C$ denote the Dunkl kernel associated with $B_q$ and multiplicity $k=(k_1, k_2),$ where $k_1$ 
and $k_2$ are the 
values of $k$ on the roots $\pm e_i$ and $\pm e_i\pm e_j$, respectively. Here $k$ belongs to a regular multiplicity set $K^{reg}\subset \mathbb C^2$ which contains those
 $k$ with $\Re k\geq 0, $ i.e. $\Re k_i\geq 0$. 
The associated Bessel function  is given by
$$ J_k^B(z,w) = \frac{1}{|G|} \sum_{g\in G} E_k^B(z,gw).$$ It is $G$-invariant in both arguments and satisfies $J_k^B(\lambda z,w) = J_k^B(z,\lambda w)$ for all $\lambda \in \mathbb C$.
If $\Re k\geq 0$, then by  \cite{dJ}, 

\begin{equation}\label{Dunklestim} |J_k^B(i\xi,\eta)|\le \sqrt{|G|}\quad \text{ for all } \xi,\eta\in \mathbb R^q.\end{equation}

\begin{proof}[Proof of Lemma \ref{Besselestim}.]
Let $x\in \overline\Omega$ with eigenvalues $\xi=(\xi_1, \ldots, \xi_q)\in \mathbb R^q$ and suppose that  $\Re \mu\geq \mu_0 +1/2.$ According to Corollary 4.6 of \cite{R2},
$$ \mathcal J_\mu(x^2) = J_k^B(2i\xi,{\bf 1}) \quad \text{ with }\, k= (\mu-\mu_0-1/2,d/2), {\bf 1} =(1,\ldots, 1).$$
Estimate \eqref{Dunklestim} implies the stated estimate of $\mathcal J_\mu(x)$ with $x\in \overline\Omega$. By the $K$-invariance of $\mathcal J_\mu$ it extends to all $x\in V.$

\end{proof}

Theorem \ref{injectivity} together with the integral representation of Theorem  \ref{int-rep-classical}
can be used to derive the following composition result for beta measures.

\begin{lemma}\label{connection-beta}
 Let $\mu,\nu_1,\nu_2\in \mathbb C$ with $\Re \mu > 2\mu_0 +1$ and $ \Re\nu_i > \mu_0$.
 Then for the mapping
$$C:\Omega_e\times \Omega_e\to\Omega_e,\quad (r,s)\mapsto P(\sqrt s\,)r$$
the push forward (or image measure)
$$\beta_{\mu,\nu_1}\circ
\beta_{\mu+\nu_1,\nu_2}:=C(\beta_{\mu,\nu_1}\otimes
\beta_{\mu+\nu_1,\nu_2})\in M_b(\Omega_e)$$
satisfies
$$\beta_{\mu,\nu_1}\circ\beta_{\mu+\nu_1,\nu_2}=\beta_{\mu,\nu_1+\nu_2}.$$
\end{lemma}

\begin{proof} We recall that for $r\in \Omega$, $P(r)$ is a positive operator and contained in $G=G(\Omega)_0.$ (The latter follows from Propos. III.2.2. of \cite{FK} and the continuity of $P$).
Thus for $r,s\in \Omega_e$, we have
$0 <  P(\sqrt s\,) r <  P(\sqrt s\,)e =s < e$, which confirms that   $C(r,s)\in  \Omega_e$.  
By Theorem  \ref{int-rep-classical} we obtain for $r\in \Omega$ 
\begin{align}
\mathcal J_{\mu+\nu_1+\nu_2}(r)&=\int_{ \Omega_e} \mathcal J_{\mu+\nu_1}(P(\sqrt s\,)r)
\> d\beta_{\mu+\nu_1,\nu_2}(s)  \notag\\
&=\int_{ \Omega_e}\int_{ \Omega_e} \mathcal J_{\mu}\bigl(P(\sqrt t\,)P(\sqrt s\,)r\bigr)d\beta_{\mu, \nu_1}(t)d\beta_{\mu+\nu_1,\nu_2}(s). \notag
\end{align}
On the other hand, 
\begin{align*} \int_{\Omega_e} \!\mathcal J_\mu\bigl(P(\sqrt r)y\bigr) & d\beta_{\mu,\nu_1}\circ\beta_{\mu+\nu_1, \nu_2}(y) = \\
    &= \int_{\Omega_e}\!\int_{\Omega_e} \!\mathcal J_\mu\bigl(P(\sqrt r)P(\sqrt s) t\bigr)d\beta_{\mu,\nu_1}(t)d\beta_{\mu+\nu_1, \nu_2}(s).\end{align*}
Now consider the argument of $\mathcal J_\mu$. By the polar decomposition of $G$  (Thm. III.5.1 of \cite{FK}),  
  there exist $k\in K$ and $x\in \Omega$ such that $P(\sqrt r) P(\sqrt s) P(\sqrt t) = P(x)k$ and therefore
  $$ P(\sqrt r) P(\sqrt s) t = P(\sqrt r) P(\sqrt s) P(\sqrt t)e = P(x)ke = P(x)e = x^2$$ and  
  $$ P(\sqrt t) P(\sqrt s) r = P(\sqrt t) P(\sqrt s) P(\sqrt r)e  = (P(x)k)^*e = k^{-1}P(x)e = k^{-1}x^2.$$
  As $\mathcal J_\mu$ is $K$-invariant, we obtain 
 $$  \mathcal J_{\mu+\nu_1+\nu_2}(r) = \int_{\Omega_e} \!\mathcal J_\mu\bigl(P(\sqrt r)y\bigr) d\beta_{\mu,\nu_1}\circ\beta_{\mu+\nu_1, \nu_2}(y).$$

If we compare this with Theorem  \ref{int-rep-classical} and use Theorem
\ref{injectivity}(1), the result follows.
\end{proof}

We now turn to the distributional extension of beta measures on matrix cones. We shall  apply Lemma \ref{Sokal}  to  the Jordan algebra $V,$  the relatively
compact set $U:= \Omega_e$, 
$\lambda=\nu$ and  the densities
\begin{equation}\label{f_nu}
f_\nu(x):= \frac{\Gamma_\Omega(\mu+\nu)}{
\Gamma_\Omega(\mu)\Gamma_\Omega(\nu)} \Delta(x)^{\mu-n/q}\Delta(e-x)^{\nu-n/q}
\end{equation}
 of the beta measures $\beta_{\mu,\nu}$ 
 from (\ref{def-beta-dist}) on $U$, where the index $\mu$ is suppressed.
We consider  the open half planes
$$E_k:=\{\nu\in   \mathbb C:\> \Re\nu> \mu_0-k\}, \quad k\in \mathbb N_0.$$
Note that $E_0\subset E_k \subset E_{k+1}$. It is clear 
that for fixed $\mu$ with $\,\Re\mu > \max(\mu_0,k)$ and $x\in U$, the function 
$\nu\mapsto f_\nu(x)$ 
is analytic on $E_k$. Moreover, by Lemma \ref{Sokal}(1), the mapping
\begin{equation}\label{measure_dist}E_0\,\to\,\cal D^{\prime}(V),
\quad \nu\mapsto 
\beta_{\mu,\nu}\end{equation}
is analytic for fixed $\mu$ with $\Re \mu > \mu_0$. 
 In order to apply the approach of Sokal \cite{S} and Lemma \ref{Sokal}(2),
we  construct 
 distributions $\beta_{\mu,\nu}\in\cal D^\prime(V)$
for $\nu\in E_k$. We here
 use ideas of Gindikin \cite{G1}, \cite{G2} for Riesz distributions;
 see  Ch.VII of \cite{FK}. 

\begin{theorem}\label{distribution-fortsetzung} 
Fix $k\in \mathbb N_0$ and an index $\mu\in \mathbb C$ with $\Re\mu > \mu_0 +kq +1$.
\begin{enumerate}
 \item[\rm{(1)}] For  $\nu\in E_k$ there
exists a unique  distribution $\beta_{\mu,\nu}\in\cal D^{\prime}(V)$
  such that the mapping
$$E_k\to\cal D^{\prime}(V) ,\quad \nu\mapsto 
\beta_{\mu,\nu}$$
is a (weakly) analytic extension of the mapping \eqref{measure_dist} from $E_0 $ to $ E_k$. 
\item[\rm{(2)}] The distributions  $\beta_{\mu,\nu}$ from part (1) belong to $\mathcal D^{\prime kq} (V)$ and have 
compact support which is contained in $\overline\Omega_e$. 
In particular,
$\beta_{\mu,\nu}(\varphi)$ is well-defined for each
  $\varphi\in \mathcal E(V)$ and $\nu\to \beta_{\mu, \nu}(\varphi)$ is 
  analytic on $E_k$ for fixed $\varphi\in \mathcal E(V)$.
  \item[\rm{(3)}] For each $\nu\in E_k$, the Bessel function $\mathcal J_{\mu+ \nu}$ satisfies
 \begin{equation}\label{Bessel_rep}
\mathcal J_{\mu+\nu}(r)=\beta_{\mu,\nu} (\mathcal J_\mu^{r}\,) \quad \text{for all }\, r\in \overline\Omega.\end{equation}
\end{enumerate}
\end{theorem}

\begin{proof}
We first note that for 
$m\in \mathbb N_0$ and  $\alpha\in\mathbb C$ with $\Re\alpha>  \mu_0 + m+1 = m+ n/q,$ the function on $V$ defined by
$$g_\alpha(x):= \left\{\begin{array}{rr} 
 \Delta(x)^{\alpha-n/q} & \quad \text{for}\quad x\in \Omega \\     
    0& \quad \text{otherwise} \end{array}\right. $$
 is contained in $C^{m}(V)$. Moreover, 
by Proposition VII.1.4 and the arguments on p.~133 of \cite{FK},
 the functions $g_\alpha$ are related to the linear differential operator
 $\Delta\bigl(\frac{\partial}{\partial x}\bigr)$ of order $q$ via
\begin{equation}\label{diff-op-rel}
\Delta\Bigl(\frac{\partial}{\partial x}\Bigr)g_\alpha =\frac{\Gamma_\Omega(\alpha)}{\Gamma_\Omega(\alpha-1)}g_{\alpha-1}\,.
\end{equation}
This leads to part (1) as follows: The case $k=0$ is trivial.
For $k\ge1$ and $\nu\in E_k$ 
 we  define a distribution $\beta_{\mu,\nu}\in \mathcal D^\prime(V)$  by 
\begin{equation}\label{cont-beta-dist}
 \beta_{\mu,\nu}(\varphi):= 
\frac{\Gamma_\Omega(\mu+\nu)}{\Gamma_\Omega(\mu)\Gamma_\Omega(\nu+k)}
\int_{V}
\Delta\Bigl(\frac{\partial}{\partial x}\Bigr)^k
\bigl(\varphi(x)g_{\mu}(x)\bigr) \cdot g_{\nu+k}(e-x)\, dx.
\end{equation}
Notice for this definition that $g_{\mu}\in C^{kq}(V)$  by our 
assumptions. Moreover, the above expression is
 analytic in $\nu\in E_k$.
It is now easy to see from (\ref{diff-op-rel}) that definition \eqref{cont-beta-dist} is consistent
with \eqref{measure_dist}. Indeed, for $\Re{\nu}+k > \mu_0 + qk +1$ we may carry out integration by parts. As
\begin{equation}\label{diff_formel} \Delta\Bigl(-\frac{\partial}{\partial x}\Bigr)^k
g_{\nu+k}(e-x) \, = \, \frac{\Gamma_\Omega(\nu + k)}{\Gamma_\Omega(\nu)} g_\nu(e-x), \end{equation}
we obtain that \eqref{cont-beta-dist} coincides with the beta measure $\beta_{\mu,\nu}$ for such $\nu$, and by analyticity with respect to $\nu$, it coincides for all $\nu \in E_0$. 
Part (2) is clear from formula (\ref{cont-beta-dist}). Finally, 
  identity \eqref{Bessel_rep} holds for all $\nu \in E_0$ according to  
  Theorem \ref{int-rep-classical},  and as both sides are analytic in $\nu \in E_k$, it extends to all $\nu\in E_k$. This  proves part (3).
\end{proof}

Similar to Riesz distributions in Theorem VII.2.2 of \cite{FK}, one can extend
analytic relations for the beta  measures (\ref{def-beta-dist})
to distributions with parameters $\mu,\nu$ as in Theorem
\ref{distribution-fortsetzung}. For instance, (\ref{def-beta-dist})
immediately leads to:

\begin{lemma}\label{prod-relation} Let $\mu,\nu\in\mathbb C$ be as in Theorem
\ref{distribution-fortsetzung} for some  $k\in \mathbb N_0$. Then
\begin{equation}\label{beta_1}\Delta(e-x)\cdot \beta_{ \mu,\nu}= 
\Bigl(\prod_{j=0}^{q-1} \frac{\nu-jd/2}{\mu+\nu-jd/2}\Bigr)\cdot \beta_{
  \mu,\nu+1}\,,\end{equation}
$$\Delta(x)\cdot \beta_{\mu,\nu}= 
\Bigl(\prod_{j=0}^{q-1} \frac{\mu-jd/2}{\mu+\nu-jd/2}\Bigr)\cdot \beta_{
  \mu+1,\nu}\,.$$
\end{lemma}

The following result concerning the existence of Sonine representations is an immediate consequence of Theorem \ref{injectivity}, Lemma \ref{Besselestim} and Theorem \ref{distribution-fortsetzung}(3).

\begin{corollary}\label{Sonineequiv} Let $k\in \mathbb N_0$ and $\Re \mu > \max(\mu_0+kq+1, 2\mu_0+1)$. Then for $\nu\in E_k$, the following are equivalent:\parskip=-1pt
 \begin{enumerate}\itemsep=-1pt
  \item[\rm{(1)}] The distribution $\beta_{\mu,\nu}$ is a complex measure.
  \item[\rm{(2)}] There exists a bounded complex measure $\beta\in M_b(\overline\Omega)$ such that $\mathcal J_{\mu+\nu}$ has the Sonine representation
  $$ \mathcal J_{\mu+\nu}(r) = \int_{\overline\Omega} \mathcal J_\mu\bigl(P\sqrt{s})\,r)d\beta(s) \quad \text{for all }\, r\in \overline\Omega.$$
 \end{enumerate}
In this case, the measure $\beta$ in (2) is unique and given by $\beta=\beta_{\mu,\nu}.$ 
\end{corollary}

We now investigate for which $\nu\in E_k$ the distribution $\beta_{\mu,\nu}$
(with $\Re\mu > \mu_0 + kq +1$) is actually a complex measure, i.e. contained in $\cal D^{\prime0}(V),$ or  even  a  positive measure. 

  It is well known (see Section VII.3 of \cite{FK}) that the Riesz distributions, 
which are given for 
$\Re{\alpha} > \mu_0$ by 
$$ R_\alpha(\varphi) = \frac{1}{\Gamma_\Omega(\alpha)}\int_{V} \varphi(x)g_\alpha(x)\,dx,$$
have a (weakly) analytic extension to distributions $R_\alpha$ for all $\alpha \in \mathbb C$. They are tempered and supported in $\overline\Omega.$
Moreover, the distribution $R_\alpha$ is a positive measure exactly if $\alpha $ belongs to the Wallach set
$$ \big\{0,\frac{d}{2},\ldots,(q-1)\frac{d}{2}= \mu_0\}\, \cup\, ]\mu_0,\infty[\,.$$
A simple proof for the necessity of this condition is given in \cite{S}. By the same method, it is also shown there that 
$R_\alpha$ is a locally finite complex Borel measure exactly if
$\alpha$ belongs to the set 
$$ W_{q,d}:= \big\{0,\frac{d}{2},\ldots,(q-1)\frac{d}{2}\big\}\, \cup\, E_0.$$
The following sufficient condition for beta distributions is
a consequence of the known results for Riesz distributions.

\begin{theorem}\label{suff1}
Let $k\in \mathbb N_0$, $\Re\mu > \mu_0+kq+1$, and let
$\nu\in E_k\cap W_{q,d}$.
Then  $\beta_{\mu,\nu}$ belongs to $\cal D^{\prime 0}(V),$ i.e. $\beta_{\mu, \nu}$ is a compactly supported complex Borel measure. In particular, 
$\beta_{\mu,0}=\delta_{e}$, provided that $0\in E_k$.

If in addition $\mu$ and $\nu$ are real, then $\beta_{\mu,\nu}$ is a probability measure. 
\end{theorem}

\begin{proof} For the normalization, recall from Theorem \ref{distribution-fortsetzung}   
that $\nu \to \beta_{\mu, \nu}(1)$ is analytic on $E_k$.
Therefore $\beta_{\mu,\nu}(1) = 1$ for all $\nu \in E_k.$

 Now let $\nu \in E_k\cap W_{q,d}$. Then 
the distribution
$$ \Delta\bigl(\frac{\partial}{\partial x}\bigr)^k g_{\nu + k}\, = 
\Gamma_\Omega(\nu +k) R_\nu\,$$
is  a locally finite complex Borel measure.
We claim that for $\varphi\in \mathcal D(V),$ 
\begin{equation}\label{beta_identity}\beta_{\mu,\nu}(\varphi) = \frac{\Gamma_\Omega(\mu+\nu)}{\Gamma_\Omega(\mu)}\cdot R_\nu^\theta(\varphi g_\mu),
 \end{equation}
where $R_\nu^\theta$ denotes the image measure (pushforward) of the Riesz measure $R_\nu$ under the mapping $\theta: V\to V, x\mapsto e-x$. 
Indeed, for $\psi\in \mathcal D(V)$ we have
\[ \Gamma_\Omega(\nu + k)R_\nu^\theta (\psi) = \Bigl(\Delta\bigl(\frac{\partial}{\partial x}\bigr)^k g_{\nu+k}\Bigr) 
(\psi\circ \theta) = \int_{V} \Delta\bigl(\frac{\partial}{\partial x}\bigr)^k \psi(x) \cdot g_{\nu+k}(e-x) dx.\]
An approximation argument shows that this identity also holds for $\psi\in C_c^{kq}(V),$ as we may approximate $\psi$ by a net $(\psi_\epsilon)_{\epsilon>0}
\subseteq \mathcal D(V)$ such that 
$\,\partial^\alpha\psi_\epsilon \to \partial^\alpha \psi$ uniformly on $V$ for all $|\alpha|\leq kq$ and the supports of the $\psi_\epsilon$ stay in a fixed relatively compact neighborhood of $\text{supp}\,\psi$.
Putting $\psi = \varphi g_\mu \in C_c^{kq}(V)$ and using formula \eqref{cont-beta-dist}, we thus obtain \eqref{beta_identity}. 
From identity \eqref{beta_identity} it is now obvious  that $\beta_{\mu,\nu}$ is a complex measure which is even positive if $\mu, \nu$ are real. 
As $R_0 = \delta_0$, it is also  immediate
that $\beta_{\mu, 0} = \delta_e$. 

\end{proof}

\begin{remark} (1)\enskip
The supports of the Riesz measures $R_\nu$ 
 with  
$\nu \in W_{q,d}\,$ are known (see Propos.VII.2.3 of \cite{FK}).
 Identity (\ref{beta_identity}) then easily gives the supports 
of the corresponding measures $\beta_{\mu, \nu}.$ In particular, $\beta_{\mu,\nu}$ is a point measure only if $\nu = 0.$

(2) \enskip Theorem \ref{suff1} is in accordance with Proposition \ref{group-case} in the group cases for $\mu$ 
 sufficiently large.
It is not clear whether for small parameters  $\mu= pd/2$ and $\nu=\widetilde pd/2$,
the probability measures $\widetilde\beta_{\mu,\nu}$
from Proposition \ref{group-case} can be obtained as distributions via analytic extension as above.
Nevertheless, we shall from now on denote  $\widetilde\beta_{\mu,\nu}$ by $\beta_{\mu,\nu}$.

\end{remark}

We are now aiming at necessary conditions on the indices under which the beta distributions $\beta_{\mu, \nu}$ on a symmetric cone are actually measures. 
Such conditions will also imply that
 the existence of an integral representation as in the above corollary 
requires non-trivial restrictions on the indices of the Bessel functions involved. As a first step, 
we extend Lemma \ref{connection-beta} for beta measures 
to a larger set of parameters for which the involved beta distributions are
measures according to Theorem \ref{suff1} or Proposition \ref{group-case}. 
The same proof as in Lemma 
\ref{connection-beta} implies:

\begin{lemma}\label{connection-beta-cont}
Let $\Re\mu> 2\mu_0+1$ and 
$\nu_1,\nu_2\in \mathbb C$ be 
such that the beta measures $\beta_{\mu,\nu_1},$
$\beta_{\mu+\nu_1,\nu_2},\,\beta_{\mu,\nu_1+\nu_2} $ exist. Then, in notation of
 Lemma \ref{connection-beta},
$$\beta_{\mu,\nu_1}\circ\beta_{\mu+\nu_1,\nu_2}:= C(\beta_{\mu,\nu_1}\otimes\beta_{\mu+\nu_1,\nu_2})\,
=\beta_{\mu,\nu_1+\nu_2}.$$
\end{lemma}

We do not know whether it is possible to derive a converse statement of Theorem \ref{suff1}
 by following
the approach of Gindikin for Riesz distributions; see
Section VII.3 of \cite{FK}. We use a different approach by Sokal \cite{S} (specifically, Lemma \ref{Sokal}), by which we
easily obtain the following result:

\begin{theorem}\label{not1}
Let $k\in \mathbb N_0$, $\Re \mu > \mu_0 +kq+1$, and
$\nu\in E_k$. If   $\beta_{\mu,\nu}\in\cal D^{\prime0}(H_q)$,
i.e.,  $\beta_{\mu,\nu}$ is a complex measure, 
then 
$$\nu\in \Bigl(\bigl\{0, \frac{d}{2}, \ldots, (q-1)\frac{d}{2}\big\} -\mathbb N_0\Bigr)\, \cup E_0\,.$$
In particular, $\nu+l\in W_{q,d}$ for some $l\in \mathbb N_0$.
\end{theorem}

\begin{proof}
We apply Lemma \ref{Sokal}(2) to $D_0:=E_0$, $D:=E_k$, and $U=\Omega_e$ and
obtain that the beta density $f_\nu$ given by \eqref{f_nu} belongs to 
$ L^1_{loc}(\overline{\Omega_e})$. It is well-known that
$$x\mapsto \Delta(x)^{\mu-n/q}\Delta(e-x)^{\nu-n/q}$$
is contained in $L^1_{loc}(\overline{\Omega_e})$ precisely for $\nu\in E_0$; see
for instance Lemma 3.4 of \cite{S}. Therefore either  $\nu\in E_0$ or
$$\frac{\Gamma_\Omega(\mu+\nu)}{\Gamma_\Omega(\mu)\Gamma_\Omega(\nu)}=0,$$
where the latter just means that $\Gamma_\Omega$ has a pole in $\nu$.
\end{proof}

We conjecture that under the conditions of Theorem \ref{not1}, it should be even true, similar as for Riesz distributions, that $\nu\in W_{q,d}.$
Our next statement confirms this conjecture
 under the assumption that $d\in\{1,2\}.$ This covers the important case of the matrix cones $\Omega_q(\mathbb F)$ over $\mathbb F=\mathbb R$ or $\mathbb C$,
 as well as the Lorentz cones 
 $\Lambda_3$ and $\Lambda_4$. 

\begin{theorem}\label{reell_komplex_Lemma}
Suppose that $d\in\{1,2\}$. Let $k\in \mathbb N$, $\Re \mu > \mu_0 +kq+3/2$ and
$\nu\in E_k$. 
If   $\beta_{\mu,\nu}$ is a complex measure, then $\, \Re\nu \geq 0.$ 
If in addition $\mu$ is real and $\beta_{\mu,\nu}$ is a positive measure, then $\,\nu\in[0,\infty[$.

\end{theorem}

\begin{proof} Notice first in the present situation, $\mu_0+kq+3/2 > 2 \mu_0 +1$. Now suppose that $\beta_{\mu,\nu}$ is a complex measure. In view of Theorem \ref{not1} it  suffices to
  consider $\nu=rd/2 -l$ with $r=0,\ldots,q-1$ and $l>0$ an integer.
  We may also assume that $\mu_0-k < \nu \leq \mu_0-k+1,$ and therefore
  $\nu = \mu_0-k +\alpha $ with $\alpha\in \{1/2, 1\}.$
  We now assume that $\nu <0$ and claim that 
  $\beta_{\mu+\nu, -\nu}$ is a complex measure. 
  In fact, our assumptions imply that 
  $$ \Re(\mu + \nu) > (\mu_0+kq +3/2)  + (\mu_0 -k+ 1/2) = 2\mu_0 +k(q-1) + 2\, > \mu_0\,.$$
  If $\,-\nu >\mu_0$,
   our claim is obvious.  Let us consider the case $\,-\nu\leq \mu_0$. Then   $\,-\nu\in E_{2\mu_0 +2-k}$ with
  $2\mu_0+2-k \in \mathbb N$. As $\nu<0$ and $\mu_0, \nu\in \frac{1}{2}\mathbb N_0$,  it follows that  $k\geq \mu_0 +1$ and
 therefore
 $$ \Re(\mu + \nu) >   2\mu_0 +k(q-1) + 2 \geq 2\mu_0 + (\mu_0+1)(q-1) +2  \geq \,\mu_0 + q(2\mu_0+2-k) +1\,.$$
  Moreover, as $0<-\nu \leq \mu_0= (q-1)d/2$ and $d=1$ or $d=2$, we conclude that
  $\,-\nu \in \{d/2, \ldots, (q-1)d/2\}\, \subset W_{q,d}.$ (Here the assumption $d\in \{1,2\}$ has been used for the first time). 
  We may now apply Theorem \ref{suff1} to the pair $(\mu+\nu,-\nu)$  and obtain again that $\beta_{\mu+\nu, -\nu}$ is a complex measure. Notice also that
  $\beta_{\mu,0}$ is a complex measure because $0\in E_k$ according to our assumptions. 
  Thus by Lemma \ref{connection-beta-cont},
  \begin{equation}\label{supports} \beta_{\mu,\nu}\circ\beta_{\mu+\nu, -\nu} = \beta_{\mu, 0} = \delta_e\,.\end{equation}
  On the other hand, the support of the measure $\,\beta_{\mu,\nu}\circ\beta_{\mu+\nu, -\nu}$      is given by 
  $$\{P(\sqrt{s}\,)r: r \in \text{supp}\,\beta_{\mu+\nu, -\nu}\,, \, s\in \text{supp}\, \beta_{\mu,\nu}\}.$$ If $P(\sqrt{s})r =e $ with
  $0\leq r,s \leq e$, then $r=s=e.$  Identity \eqref{supports} therefore implies that 
  $\,\text{supp}\,\beta_{\mu, \nu} = \text{supp} \,\beta_{\mu+\nu, -\nu} = \{e\}$, which is possible only if $\nu = 0$. 
  This contradicts our assumption and proves the first statement. 
  
 If in addition $\mu$ is real and $\beta_{\mu,\nu}$ is a positive measure, then it is clear from Theorem \ref{distribution-fortsetzung}(3) that
  $\nu$ is real. This shows the second statement.

\end{proof}
The argument above relies on the condition $d\in \{1,2\},$ and we do not know whether Theorem \ref{reell_komplex_Lemma} extends to 
larger Peirce constants. 
Let us summarize our results for $d\in \{1,2\}$. 

\begin{corollary}\label{summary1}
Suppose $\,d\in\{1,2\}$. Let $k\in \mathbb N$  and 
$\Re\mu> \mu_0 + kq + 3/2.$  Then for $\nu\in E_k$, the following statements are
equivalent:
\begin{enumerate}\itemsep=-1pt
 \item[\rm{(1)}] $\beta_{\mu,\nu}$ is a complex measure;
 \item[\rm{(2)}]  $\nu \in W_{q,d}$;
\item[\rm{(3)}] There exists a bounded complex measure $\beta\in M_b(\overline\Omega)$ such that 
$$ \mathcal J_{\mu + \nu}(r) = \int_{\overline\Omega} \mathcal J_\mu(rs) d\beta(s) \quad \text{ for all } r\in \overline\Omega\,.$$
\end{enumerate}
If $\mu$ is real with $\mu> \mu_0 + kq + 3/2,$  then for $\nu\in E_k$ the following are
equivalent:

\begin{enumerate}\itemsep=-1pt
 \item[\rm{(1)}] $\beta_{\mu,\nu}$ is a positive measure;
 \item[\rm{(2)}]  $\nu$ is contained in the Wallach set
$$\Bigl\{0,\frac{d}{2},\ldots,(q-1)\frac{d}{2}=\mu_0\Bigr\}\,\cup\, ]\mu_0,\infty[\,;$$ 
\item[\rm{(3)}] There exists a  probability measure $\beta\in M^1(|overline\Omega)$ such that 
$$ \mathcal J_{\mu + \nu}(r) = \int_{\overline\Omega} \mathcal J_\mu(rs) d\beta(s) \quad \text{ for all } r\in \overline\Omega\,.$$
\end{enumerate}
In both cases, the measure $\beta$ in $(3)$ is unique and given by  $\beta_{\mu, \nu}$. 

\end{corollary}

\begin{proof}  In both cases, 
implication  (1)$\Rightarrow $(2) follows from Theorem \ref{not1} in combination with Theorem \ref{reell_komplex_Lemma}. The remaining parts are immediate from 
Corollary \ref{Sonineequiv}, Theorem \ref{suff1} and Theorem \ref{injectivity}(2). 

\end{proof}

Corollary \ref{summary1} implies in particular that for $q>1$ and sufficiently large 
 $\mu>0$, there exist indices $\nu >0$ such that  $\mathcal J_{\mu+\nu}$
admits no positive integral representation with respect to $\mathcal J_{\mu}$. So there exists no Sonine-type formula in these cases.
This is a surprising contrast compared to the one-variable case.

\begin{remark} The Jack polynomials $C_\lambda^\alpha$ have non-negative  coefficients in their expansion with respect to the monomial symmetric functions(\cite{KS}). In view of formula \eqref{power-j},
this implies that 
\begin{equation}\label{Bessel_neg}
\mathcal J_\mu(-r) >0 \quad\text{for } \mu > \mu_0 \text{ and all } r\in\overline\Omega.\end{equation}
Similar to an argument in the
  appendix of \cite{S}, this observation together with Theorem \ref{distribution-fortsetzung}(3) and
 identity \eqref{beta_1} leads for $d=2$ to an alternative 
 proof that for $\mu>\mu_0+kq+1$ and indices $\nu\in [0,\infty[$ which do not belong to the Wallach set,
the distribution $\beta_{\mu,\nu}$ cannot be a positive measure.
In fact, otherwise identity \eqref{beta_1} would
imply that $\beta_{\mu,\nu+l}$ is a negative measure for $l=1$ or $l=2$, because  the product on the right side of formula \eqref{beta_1} will be negative for either $\nu$ or $\nu +1$. (Here $d=2$ is relevant).  On the other hand, Theorem \ref{distribution-fortsetzung}(3) immediately implies that
\[ \mathcal J_{\mu+\nu +l}(-r) = \int_{\overline\Omega_e} \mathcal J_\mu\bigl(-P(\sqrt{s})r\bigr) d\beta_{\mu,\nu+l}(s) \]
for all $r\in \overline\Omega$, in contradiction to \eqref{Bessel_neg}.

\end{remark}



\begin{thebibliography}{00}






\bibitem{AAR} G. Andrews, R. Askey, R. Roy: Special Functions. Cambridge Univ. Press, 1999.

\bibitem{A} R. Askey, Orthogonal Polynomials and Special
  Functions. Regional Conference Series in Applied Mathematics, SIAM, Philadelphia, 1975.
  
  \bibitem{Asl} H. Aslaksen, Quaternionic determinants. \textit{Math. 
Intelligencer} 18 (1996), 57--65. 
  
  \bibitem{BF} T.H. Baker, P.J. Forrester, The Calogero-Sutherland model and generalized classical
polynomials. \textit{Comm. Math. Phys.} 188 (1997), 175--216.
 
 
 \bibitem{Bie} E. Bierstone, Differentiable functions. \textit{Bol. Soc. Bras. Mat.} 11, no. 2 (1980), 139--190.


\bibitem{CL} M. Casalis, G. Letac, The Lukacs-Olkin-Rubin characterization of Wishart distributions on symmetric cones.
 \textit{Ann. Stat.} 24 (1996), 763-786.
 
 \bibitem{Co} A.G. Constantine, Some non-central distribution problems in multivariate analysis.  \textit{Ann. Math. Statist.}  34  (1963), 1270--1285.
 
\bibitem{D} J.A. Diaz-Garcia, Riesz and Beta-Riesz distributions.
 \textit{Austrian J. Statistics} 45 (2016),  35--51.


\bibitem{Di} H. Dib, Fonctions de Bessel sur une alg\'ebre de Jordan.
\textit{J. Math. Pures et Appl.} 69 (1990), 403--448.

\bibitem{dJ} M.F.E. de Jeu,  The Dunkl transform. \emph{Invent. Math.} 
113 (1993), 147 -- 162.


\bibitem{FK} J. Faraut, A. Kor\'anyi, Analysis on symmetric cones. Oxford
 Science Publications,
 Clarendon press, Oxford 1994.

\bibitem{FT} J. Faraut, G. Travaglini, Bessel functions associated with
  representations of formally real Jordan algebras. \textit{J. Funct. Anal.}
  71 (1987), 123--141.

\bibitem{Fa} R. Farrell, Multivariate Calculus: Use of Continuous
  Groups. Springer Verlag, New York 1985.
  
\bibitem{G1} S. Gindikin, Analysis on homogeneous domains. 
 \textit{Russian Math. Surveys} 19, No 4, 1-89 (1964).

\bibitem{G2} S. Gindikin, Invariant generalized functions in 
homogeneous domains. \textit{Funct. Anal. Appl. } 9, 50-52 (1975).

\bibitem{GR} K. Gross, D. Richards, Special functions of matrix argument. I: Algebraic Induction, Zonal Polynomials, and Hypergeometric Functions. \textit{Trans. Amer. Math. Soc.} 301 (1987), 781--811.

\bibitem{H} C.S. Herz, Bessel functions of matrix argument.
\textit{Ann. Math.} 61 (1955), 474--523.

\bibitem{Ho} L. H\"ormander, 
The Analysis of Linear Partial Differential Operators
I. Springer Verlag, Berlin-Heidelberg-New York, 1990.

\bibitem{Ka} J. Kaneko, Selberg integrals and hypergeometric functions
associated with Jack polynomials. \textit{SIAM J. Math. Anal.} 24 (1993),  1086--1100.


\bibitem{KS} F. Knop, S. Sahi, A recursion and a combinatorial formula for Jack polynomials. \textit{Invent. Math.} 128 (1997), 9--22.
 
\bibitem{Ko} B. Kolodziejek, Characterization of beta distribution on symmetric cones. \textit{J. Multiv. Anal.}
143 (2016), 414--423.


\bibitem{Ma1} I.G. Macdonald, Commuting differential operators and zonal
spherical functions. In: Algebraic groups (Utrecht 1986), eds. a.M. Cohen et al, \textit{Lecture Notes in Mathematics} 1271, Springer-Verlag, Berlin, 1987.

\bibitem{Ma} I.G. Macdonald, Hypergeometric functions I. 
 	arXiv:1309.4568. 
 	
\bibitem{Mo} J. M\"ollers, A geometric quantization of the Kostant-Sekiguchi correspondence for scalar type unitary highest weight representations.  \textit{Doc. Math.} 18 (2013), 
785--855.


\bibitem{Mu} R. Muirhead, Aspects of Multivariate Statistical Theory. 
Wiley, New York, 1982.

\bibitem{Na} R. Nakahama, Integral formula and upper estimate of $I$-and $J$-Bessel functions on Jordan algebras. 
\textit{J. Lie Theory} 24 (2014), 421--438.

\bibitem{Ne} Y.A. Neretin, Matrix beta integrals: an overview.
 In: P. Kielanowski et al. (eds.), Geometric methods in physics. XXXIII workshop, Bialowieża,
 Poland 2014.  Birkhäuser/Springer, Trends in Mathematics, 257-272 (2015).


\bibitem{OR} I. Olkin, H. Rubin, Multivariate beta distributions and independence poperties of the Wishart distribution. 
\textit{Ann. Math. Stat.} 35 (1964), 261--269.


\bibitem{R1} M. R\"osler, Dunkl operators: Theory and applications.
In: Orthogonal polynomials and special functions, Leuven 2002, \textit{Lecture Notes in Math.} 1817 (2003), 93--135.

\bibitem{R2} M. R\"osler, Bessel convolutions on matrix cones.
 \textit{Compos. Math.}  143 (2007), 749--779.


\bibitem{RV3} M. R\"osler, M. Voit,  
Olshanski spherical functions for infinite dimensional motion groups
of fixed rank.  \textit{J. of Lie Theory} 23 (2013), 899--920.


\bibitem{Ru} W. Rudin, Functional Analysis. McGraw-Hill, 2nd ed. 1991.

\bibitem{S} A. D. Sokal, When is a Riesz distribution a complex measure?
\textit{Bull.  Soc. Math.  France} 139 (2011), 519-534.

 \bibitem{Sr} M.S. Srivastava, Singular Wishart and multivariate beta distributions.
 \textit{ Ann. Statist.} 31 (2003), 1537-1560.

\bibitem{St} R.P. Stanley, Some combinatorial properties of Jack symmetric functions. \textit{Adv. Math.} 77 (1989), 76--115.


\bibitem{U} H. Uhlig, On singular Wishart and singular multivariate beta
  distributions. \textit{Ann. Stat.} 22 (1994), 395--405.



\end{thebibliography}
\end{document}